\documentclass[reqno,12pt]{amsart}
\usepackage{amssymb,amsthm,amsmath,amsfonts,amstext,mathrsfs}
\usepackage{latexsym}
\newtheorem*{thm}{Theorem}

\newtheorem{prop}{Proposition}
\newtheorem{defin}{Definition}
\newtheorem{quest}{Question}

\input xy
\xyoption{all}

\newcommand{\m}{\mathbf}

\title[Translation equation which arises from homothety]
{Multi-variable translation equation which arises from homothety}
\author[Giedrius Alkauskas]{Giedrius Alkauskas}
\begin{document}

\begin{abstract}In many regular cases, there exists a (properly defined) limit of iterations of a function in several real variables, and this limit satisfies the functional equation $(1-z)\phi(\m{x})=\phi(\phi(\m{x}z)(1-z)/z)$; here $z$ is a scalar and $\m{x}$ is a vector.
This is a special case of a well-known translation equation. In this paper we present a complete solution to this functional equation in case $\phi$ is a continuous function on a single point compactification of a $2-$dimensional real vector space. It appears that, up to conjugation by a homogeneous continuous function, there are exactly four solutions. Further, in a $1-$dimensional case we present a solution with no regularity assumptions on $\phi$.
\end{abstract}
\maketitle
\begin{center}
\end{center}
\begin{center}
\end{center}
\section{Introduction and main result}
\footnotetext[1]{The author gratefully acknowledges support
from the Austrian Science Fund (FWF) under the project Nr. P20847-N18.}

The following problem was given in the problem section of the {\it American Mathematical Monthly}.\\

\noindent\textbf{Problem 11149. }\it Let $g(x)=\log(1+x)$. Fix $x>0$. Find
\begin{eqnarray}
f(x):=\lim_{n\rightarrow\infty}n\cdot\underbrace{g\circ g\circ\cdots\circ g}_{n}
\Big{(}\frac{x}{n}\Big{)}.
\label{iter}
\end{eqnarray}\rm
The answer is $f(x)=\frac{2x}{x+2}$, and the straightforward solution can be found in corresponding issue of the {\it Monthly}.\\

As an introduction, let us rework this iteration more carefully. We will see that the limit function, if it exists, satisfies a certain natural functional equation.
Suppose, the limit function $f(x)$ is well defined and is continuous. We set
\begin{eqnarray*}
L_{n}(x):=n\cdot\underbrace{g\circ\cdots\circ g}_{n}
\Big{(}\frac{x}{n}\Big{)}=f(x)+\epsilon_{n}(x).
\end{eqnarray*}

Now choose any positive integers $n'$ and $n''$, set $n'+n''=n$. To avoid complicated notation, we rather
choose to write $o(1)$ instead of $\epsilon_{n}(x)$, the implied case being $n\rightarrow\infty$.
Consider $L_{nk}(x)$ for an arbitrary $k\in\mathbb{N}$:
\begin{eqnarray*}
L_{nk}(x)&=&f(x)+o(1)=nk\cdot\underbrace{g\circ\cdots\circ g}_{nk}
\Big{(}\frac{x}{nk}\Big{)}\\
&=&nk\cdot\underbrace{g\circ\cdots\circ g}_{nk}
\Big{(}\frac{xn'}{n}\cdot\frac{1}{n'k}\Big{)}\\
&=&\frac{n}{n''}\cdot n''k\cdot\underbrace{g\circ\cdots\circ g}_{n''k}
\Big{(}\frac{1}{n'k}\cdot f\Big{(}\frac{xn'}{n}\Big{)}+\frac{o(1)}{k}\Big{)}\\
&=&\frac{n}{n''}\cdot n''k\cdot\underbrace{g\circ\cdots\circ g}_{n''k}
\Big{(}\frac{1}{n''k}\cdot\frac{n''}{n'}\cdot f\Big{(}\frac{xn'}{n}\Big{)}+\frac{o(1)}{k}\Big{)}\\
&=&\frac{n}{n''}\cdot f\Big{(}f\Big{(}\frac{xn'}{n}\Big{)}\cdot\frac{n''}{n'}+o(1)\Big{)}
+o(1).
\end{eqnarray*}
Now, take the limit $k\rightarrow\infty$ in the equality
\begin{eqnarray*}
\frac{n''}{n}\cdot f(x)+o(1)=
f\Big{(}f\Big{(}\frac{xn'}{n}\Big{)}\cdot\frac{n''}{n'}+o(1)\Big{)}+
o(1).
\end{eqnarray*}
Due to continuity of $f$,  we obtain
$f(x)n''/n=f\Big{(}f(xn'/n)n''/n'\Big{)}$. Denote $n'/n$ by
$z$. Thus, we have the functional equation
\begin{eqnarray}
(1-z)f(x)=f\Big{(}f(xz)\frac{1-z}{z}\Big{)}, \text{ for any }x>0, \quad0<z<1.\label{funkk}
\end{eqnarray}
This is valid for rational $z$, but due to continuity of $f$ we obtain this for all real
$z$ in the interval $(0,1)$. Proposition \ref{prop10} in the end of this paper shows that we do not need any regularity assumption on a function $f$. Provided that $f$ maps $(0,\infty)$ to $(0,\infty)$, all solutions are given by $f(x)=\frac{x}{Cx+1}$, $C\geq 0$.\\

The natural question arises: what if we start with a function in several real variables and perform the iterations (\ref{iter})?
For example, let us choose $g(x, y)=(x-\frac{1}{2}x^2+\frac{1}{2}y^2,y-xy)$. What is the limit
\begin{eqnarray*}
\phi(x,y):=\lim_{n\rightarrow\infty}n\cdot\underbrace{g\circ g\circ\cdots\circ g}_{n}
\Big{(}\frac{1}{n}\cdot(x,y)\Big{)}.
\end{eqnarray*}
It certainly exists for, say, $0<x<1$, $0<y<1$, and it is given by the example (\ref{pvz5}) below. Yet, we do not pose a question of convergence of these iterations. We are interested whether the limit function satisfies any natural functional equation. The direct inspection of the above derivation of the functional equation shows that everything carries out to a multi-dimensional case without alterations. Thus, the equation we get is
\begin{eqnarray}
(1-z)\phi(\m{x})=\phi\Big{(}\phi(\m{x}z)\frac{1-z}{z}\Big{)},\quad \m{x}\in\mathbb{R}^{k}.\label{funk}
\end{eqnarray}
(For a moment, we do not specify for which $z$ this should hold).\\

The equation (\ref{funk}) is a special case of the equation
\begin{eqnarray*}
F(F(\alpha,x),y) = F(\alpha, x\cdot y).
\end{eqnarray*}
Here the function $F$ takes its values in a certain set $\Gamma$ and is defined on
a subset of a Cartesian product $\Gamma\times G$, where G is a set with a binary operation, denoted by
$``\cdot"$. This equation is called {\it the translation equation}. The paper \cite{moszner1} gives the summary on the results on this equation that appeared before 1973, and the exposition is continued in \cite{moszner2}. In \cite{bj} the authors are interested in finding conditions under which any Carath\'{e}odory solution of the translation or even more general functional equations (for $G=[+,(0,\infty)]$, $\Gamma$ being a metric space)
is continuous without assuming the compactness of $\Gamma$. The paper \cite{bcj} continues these investigations. The author in \cite{berg} considers the translation equation
$F(F(z, s), t)=F(z,s+t)$ for $s,t\in\mathbb{R}^{m}$, $z\in\mathbb{R}^{n}$, subject to the initial condition $F(z,0)=z$, and finds the local structure of a solution, provided its first order partial derivatives are continuous. In \cite{mach} the authors investigate the stability of the translation equation in case $G$ is a monoid with a unit. In \cite{fr} the authors are concerned with a formal translation equation and cocycle equations in the ring of formal power series $\mathbb{C}[[x]]$. The paper \cite{barbara} deals with the iterable functions. A {\it continuous iteration semigroup} on a set $I$ is  a function $F:I\times(0,\infty)\mapsto I$ which is continuous with respect to each variable and which satisfies the translation equation. Then {\it iterable functions} are functions $f:I\mapsto I$ which are embeddable into a continuous iteration semigroup; that is, $F(x,1)=f(x)$. As we will see later, the solutions to (\ref{funk}) are thus iterable functions.  For more general discussion on translation equation we refer to expository articles \cite{moszner1,moszner2}.\\

In the current paper we are dealing with the special case of the translation equation, where $F$ takes the form
$F(\m{x},z)=\frac{1}{z}\phi(\m{x}z)$.\\

After having solved a $1$-dimensional case (see Proposition \ref{prop10}), we see that one family of the solutions of (\ref{funk}) (in $2-$dimensional case) is given by
\begin{eqnarray}
\phi(x,y)=\Big{(}\frac{x}{ax+1},\frac{y}{by+1}\Big{)},\quad a,b\text{ fixed.}
\label{pvz1}
\end{eqnarray}
(In all these examples we do not specify in which region the functional equation is valid. This will be clarified later).
Some other solutions are given by
\begin{eqnarray}
\phi(x,y)=\Big{(}\frac{x}{ax+by+1},\frac{y}{ax+
by+1}\Big{)}.
\label{pvz2}
\end{eqnarray}
Another example:
\begin{eqnarray}
\phi(x,y)=\Big{(}\frac{x}{(by+1)(ax+by+1)},
\frac{y}{by+1}\Big{)}.
\label{pvz3}
\end{eqnarray}
But there are other, much more interesting solutions! For example
\begin{eqnarray*}
\phi(x,y)=\Big{(}\frac{2x^2-8y^2+x}{4x-8y+1},\frac{x^2-4y^2+y}{4x-8y+
1}\Big{)}.
\end{eqnarray*}
A second one:
\begin{eqnarray}
\phi(x,y)=\Big{(}\frac{2x^2+2y^2+4x}{x^2+y^2+4x+4},\frac{4y}{x^2+y^2+
4x+4}\Big{)}.
\label{pvz5}
\end{eqnarray}
A third one:
\begin{eqnarray*}
\phi(x,y)=\Big{(}(x-y)^{2}+x,(x-y)^{2}+y\Big{)}.
\label{pvz8}
\end{eqnarray*}
We finish with two last examples. The first is $\phi(x,y)=(\phi_{1}(x,y),\phi_{2}(x,y))$, where
\begin{eqnarray}
\phi_{1}(x,y)=\frac{4(x^4+2y^2x^2+y^4+2yx^2)^2y^2x}{(x^6+3x^4y^2+3x^2y^4+y^6+4yx^4+4y^3x^2+4y^2x^2)^2};\nonumber\\
\phi_{2}(x,y)=\frac{8(x^4+2y^2x^2+y^4+2yx^2)x^2y^4}{(x^6+3x^4y^2+3x^2y^4+y^6+4yx^4+4y^3x^2+4y^2x^2)^2}.
\label{pvz6}
\end{eqnarray}
Let us make a convention that for real $x$, $x^{1/3}=\text{sgn}(x)|x|^{1/3}$. Then the last one is given by
\begin{eqnarray}
\phi_{1}(x,y)=\frac{x^3}{x^3+x^2+y^2}+\frac{xy^2}{(x^3+x^2+y^2)^{1/3}(y^3+x^2+y^2)^{2/3}};\nonumber\\
\phi_{2}(x,y)=\frac{y^3}{y^3+x^2+y^2}+\frac{x^2y}{(x^3+x^2+y^2)^{2/3}(y^3+x^2+y^2)^{1/3}}.
\label{pvz7}
\end{eqnarray}
\indent Here in the example (\ref{pvz5}) we see that our problem is of affine rather than of projective nature. Indeed, suppose
\begin{eqnarray*}
\phi(x,y)=\phi(x:y:z)=(2x^2+2y^2+4xz: 4yz:x^2+y^2+4xz+4z^2).
\end{eqnarray*}
The special point (where all coordinates vanish) is $(x:y:z)=(-2:0:1)$. If $(x:y:z)=(x_{0}:0:1)$ and $x_{0}\rightarrow-2$, then $\phi(x:y:z)\rightarrow(1:0:0)$.
On the other hand, if $(x:y:z)=(-2:y_{0}:1)$ and  $y_{0}\rightarrow0$, then $\phi(x:y:z)\rightarrow(0:1:0)$. Thus, $\phi$ does
not extend continuously to the projective plane $\mathbb{R}P^{2}$. Rather, it is a continuous function on the compactification
$\mathbb{R}^{2}\cup\{\infty\}\sim \mathbb{S}^{2}$.\\

Let $ k\in\mathbb{N}$, and $\mathbb{R}^{ k}\cup\{\infty\}=\widehat{\mathbb{R}^{ k}}\sim\mathbb{S}^{ k}$ be a single point compactification of $\mathbb{R}^{ k}$ (in the sequel, the symbol ``$\infty$" will denote this point, as well as a point at infinity of
$\mathbb{R}$. This should not cause a confusion). First, we need a definition.
\begin{defin}
We call a continuous bijection $\ell:\widehat{\mathbb{R}^{k}}\mapsto\widehat{\mathbb{R}^{k}}$  {\rm a homothetic function}, if for all
$\m{x}\in\mathbb{R}^{k}$, $z\in\mathbb{R}$ we have
\begin{eqnarray*}
\ell(z\m{x})=z\ell(\m{x}).
\end{eqnarray*}
\end{defin}

If $Q$ is a positive quadratic form and $Q(\m{a})=1$, let us define
\begin{eqnarray}
\phi_{\m{a},Q}(\m{x})=
\frac{\m{a}Q(\m{x})+\m{x}}{Q(\m{a}+\m{x})}.
\label{bas}
\end{eqnarray}
(Note that the numerator is a vector and the denominator is a scalar). These functions (with $Q(\m{a})=1$) will constitute basic solutions of the functional equation we are dealing with.\\

Further, if $L$ is a linear form and $L(\m{c})=\m{0}$ for a certain vector $\m{c}$, let us define
\begin{eqnarray*}
\phi_{\m{c},L}(\m{x})=\m{c}L^{2}(\m{x})+\m{x}.
\end{eqnarray*}
\textbf{Important remark. }Throughout the paper, we use $k$ to denote the dimension of the space $\mathbb{R}^{k}$. All our subsequent results are valid for all $k\geq1$, except that the affirmative resolution of Question \ref{que} is mathematically rigorous only for $k=1$ and $k=2$. Hence, our main result is valid only for $k=1$ and $k=2$ (fortunately, the latter case is still very non-trivial). Henceforth we assume $k=2$, but we rather choose to use the unspecified index $k$ to emphasize that all results (apart from the main Theorem) hold for $k>2$ as well.
\begin{thm}Let $k=2$. Suppose, a continuous function
$\phi:\widehat{\mathbb{R}^{ k}}\mapsto\widehat{\mathbb{R}^{ k}}$, $\phi(\infty)=\m{a}$,
satisfies the functional equation (\ref{funk}) for any $\m{x}\in\widehat{\mathbb{R}^{ k}}$, $z\in\widehat{\mathbb{R}}$.\\
\begin{itemize}
\item Assume $\m{a}\in\mathbb{R}^{ k}\setminus\{\m{0}\}$. Then there exists a positive definite quadratic form $Q$ for which $Q(\m{a})=1$, and a homothetic function $\ell$ for which
$\ell(\m{a})=\m{a}$, such that
\begin{eqnarray}
\phi(\m{x})=\ell^{-1}\circ\phi_{\m{a},Q}\circ\ell(\m{x}).
\label{sol}
\end{eqnarray}
\item Assume $\m{a}=\infty$. Then there exists a linear form $L$ and a vector $\m{c}$ for which $L(\m{c})=0$, and a homothetic function $\ell$, such that
\begin{eqnarray*}
\phi(\m{x})=\ell^{-1}\circ\phi_{\m{c},L}\circ\ell(\m{x}).
\end{eqnarray*}
\end{itemize}
\label{th1}
\end{thm}
Suppose, the function $\phi$ satisfies (\ref{funk}). Substitution $z=1$ implies $\phi(\m{0})=\m{0}$. Suppose also, for a certain $\m{x}\neq\m{0}$, $\phi(\m{x})=\m{0}$. Then the functional equation gives (we see it after a substitution $\m{x}\rightarrow\m{x}/z$) that $\phi(\m{x}z)=\m{0}$, $z\in\widehat{\mathbb{R}}$. In particular, $\phi(\infty)=\m{0}$. Now in (\ref{funk}) take a limit, as $z\rightarrow\infty$. This shows that for every $\m{x}$ there exists $\lim_{z\rightarrow\infty}(1-z)\phi(\m{x})=\m{0}$. Of course, this can happen only iff $\phi(\m{x})\equiv\m{0}$. Thus, if this is not the case,
\begin{eqnarray}
\phi(\m{x})\neq\m{0}\text{ for }\m{x}\neq\m{0}.
\label{nulis}
\end{eqnarray}
Consequently, later we will show that our main Theorem can be formulated as follows.
\begin{thm}{(\it Alternative formulation.) } Let $k=2$, $\m{x}\in\widehat{\mathbb{R}^{ k}}$, $z\in\widehat{\mathbb{R}}$. All continuous in $\widehat{\mathbb{R}^{k}}$ solutions to (\ref{funk}) are given by:
\begin{itemize}
\item $\phi_{{\tt id}}(\m{x})=\m{x}$;
\item $\phi_{0}(\m{x})=\m{0}$;
\item (Case $\m{a}$ is finite) $\ell^{-1}\circ\phi_{1}\circ\ell(\m{x})$, where $\ell$ is a homothetic function, and the $j-$th coordinate of $\phi_{1}$ is given by
\begin{eqnarray}
(\phi_{1}(\m{x}))_{j}=\frac{\sum\limits_{i=1}^{k}x^{2}_{i}+k\cdot x_{j}}
{\sum\limits_{i=1}^{k}(x_{i}+1)^2},\quad j=1,\ldots,k;
\label{basic}
\end{eqnarray}
\item (Case $\m{a}=\infty$) $\ell^{-1}\circ\phi_{\infty}\circ\ell(\m{x})$, where $\ell$ is a homothetic function, $\m{d}=(d_{1},d_{2},\ldots,d_{k})$ is a fixed in advance non-zero vector such that $\sum_{i=1}^{k}d_{i}=0$, and the $j-$th coordinate of $\phi_{\infty}$ is given by
\begin{eqnarray*}
(\phi_{\infty}(\m{x}))_{j}=d_{j}\Big{(}\sum\limits_{i=1}^{k}x_{i}\Big{)}^{2}+x_{j},\quad j=1,\ldots,k.
\end{eqnarray*}
\end{itemize}
\end{thm}
\indent{\it Remark 1. }In the sequel we will concentrate to the case $\m{a}$ is finite. The infinite case is analogous. Of course, the first case of the above alternative formulation is covered in the last if $\m{d}=\m{0}$. Though for the sake of lucidity it is better to separate it, since the above four cases present the complete list of solutions up to conjugation by a homothetic function.\\
\indent{\it Remark 2. }The function $\phi_{\m{a},Q}(\m{x})$, where $Q(\m{a})=1$, can be given the following expression.
Let $B(\m{x},\m{y})=Q(\m{x}+\m{y})-Q(\m{x})-Q(\m{y})$ be the associated
bilinear form. Then
\begin{eqnarray}
\phi_{\m{a},Q}(\m{x})=\Big{(}\m{a}Q(\m{x})+\m{x}\Big{)}\cdot\Big{(}Q(\m{x})\cdot Q(\m{a})+B(\m{x},\m{a})+1\Big{)}^{-1}.
\label{alter}
\end{eqnarray}
This expression is less elegant though more convenient, since it is merely required that $Q(\m{a})\neq 0$.
The last expression for the function $\phi_{\m{a},Q}(\m{x})$ does not alter after a substitution $\m{a}\rightarrow c\m{a}$, $Q\rightarrow c^{-1}Q$, $c\in\mathbb{R}\setminus\{0\}$: $\phi_{\m{a},Q}=\phi_{c\m{a},c^{-1}Q}$. Thus, for a specific $c$ this expression reduces to the one given in Theorem.\\
\indent{\it Remark 3. }Let $ k=2$. Instead of looking for solutions in $\widehat{\mathbb{R}^{2}}\sim\mathbb{S}^2$, we can consider other spaces. Thus, for example, the solution (\ref{pvz1}) is a continuous function (and takes values) on $\mathbb{S}^{1}\times\mathbb{S}^{1}$, which is a torus rather than a sphere. On the other hand, the solution (\ref{pvz2}) is continuous (and takes values) on a projective plane $\mathbb{R}P^2$. Finally, the solution (\ref{pvz3}) is not even continuous on $\mathbb{S}^{2}$. These cases are not covered by the Theorem. Nevertheless all quadratic forms
$Q$ produce solutions, as given by (\ref{alter}). For example, the case (\ref{pvz1}) occurs when $Q(x,y)=xy$.
\section{The proof}
We will deal with the case $\m{a}$ is finite. For arbitrary $\phi$, let
\begin{eqnarray*}
\phi^{z}(\m{x})=\frac{1}{z}\cdot\phi(\m{x}z).
\end{eqnarray*}
Then the functional equation can be rewritten in the form
\begin{eqnarray}
\phi^{z_{1}}\circ\phi^{z_{2}}(\m{x})=\phi^{z_{1}+z_{2}}(\m{x}),\quad, z_{1},z_{2}\in\mathbb{R},\quad z_{1},z_{2},z_{1}+z_{2}\neq 0.\label{fuo}
\end{eqnarray}
We have:
\begin{eqnarray}
(\gamma\circ\chi)^{z}(\m{x})=\frac{1}{z}\cdot\gamma\Big{(}\chi(\m{x}z)\Big{)}=
\frac{1}{z}\cdot\gamma\Big{(}z\cdot\frac{1}{z}\cdot\chi(\m{x}z)\Big{)}=(\gamma^{z}\circ\chi^{z})(\m{x}).
\label{trans}
\end{eqnarray}
The following proposition is immediate.
\begin{prop}
If $\chi$ is a solution to (\ref{funk}), then so is $\chi^{z}$. In general, let $\ell$ be a homothetic function. Then $\chi^{\ell}:=\ell^{-1}\circ\chi\circ \ell$ is a solution to (\ref{funk}).
\label{prop1}
\end{prop}
\begin{proof}Of course, the first claim of the proposition is a special case of the second in case $\ell$ is given by a scalar matrix $z\cdot I$, where $I$ is the identity matrix. For $z_{1},z_{2},z_{1}+z_{2}\neq 0$, we thus have:
\begin{eqnarray*}
(\chi^{\ell})^{z_{1}}\circ(\chi^{\ell})^{z_{2}}(\m{x})&=&
\frac{1}{z_{1}}\cdot\ell^{-1}\circ\chi\circ\ell\Big{(}\frac{z_{1}}{z_{2}}\cdot\ell^{-1}\circ\chi\circ\ell(z_{2}\m{x})\Big{)}\\
&=&\ell^{-1}\Big{[}\frac{1}{z_{1}}\cdot\chi\Big{(}\frac{z_{1}}{z_{2}}\cdot\chi\circ\ell(z_{2}\m{x})\Big{)}\Big{]}=
\ell^{-1}\Big{[}\frac{1}{z_{1}+z_{2}}\cdot\chi\Big{(}\ell(\m{x})(z_{1}+z_{2})\Big{)}\Big{]}\\
&=&(\chi^{\ell})^{z_{1}+z_{2}}(\m{x}),
\end{eqnarray*}
and we are done.\end{proof}
{\it Examples. }Thus, let us put $k=2$, and consider (for $x\neq 0$, $y\neq 0$, $(x,y)\neq\infty$)
\begin{eqnarray*}
\ell(x,y)=\Big{(}\frac{x^{2}+y^{2}}{y},\frac{x^{2}+y^{2}}{x}\Big{)}.
\end{eqnarray*}
The inverse is given by (for $(x,y)\neq (0,0)$ or $\infty$)
\begin{eqnarray*}
\ell^{-1}(x,y)=\Big{(}\frac{x^{2}y}{x^{2}+y^{2}},\frac{xy^{2}}{x^{2}+y^{2}}\Big{)}.
\end{eqnarray*}
If $\phi$ is as in the example (\ref{pvz5}), then $\ell^{-1}\circ\phi\circ\ell$ is given by (\ref{pvz6}). Of course, in this case $\ell$ is not a bijection. The function $\phi$ thus constructed does indeed satisfy (\ref{funk}). The denominator of (\ref{pvz6}) vanishes only for $(x,y)=(0,0)$. Therefore, all the requirements of the Theorem are satisfied, only $\phi(x,y)$ fails to be continuous at $(x,y)=(0,0)$ and $(x,y)=\infty$.
As a matter of fact, all rational functions which satisfy the requirements of the Theorem are given by $\m{0}$, $\phi_{\m{a},Q}(\m{x})$ and $\phi_{\m{c},L}(\m{x})$, since all homothetic bi-rational maps $\widehat{\mathbb{R}^{k}}\mapsto\widehat{\mathbb{R}^{k}}$ are given by non-degenerate linear transformations.\\
\indent Consider another example which, too, fails only the continuity requirement. Let
\begin{eqnarray*}
\ell(x,y)=\Big{(}\frac{x^3}{x^2+y^2},\frac{y^3}{x^2+y^2}\Big{)}.
\end{eqnarray*}
This is a homothetic function. It maps the unit circle to the astroid $|x|^{2/3}+|y|^{2/3}=1$. If, as before, we make a convention that for real $x$, $x^{1/3}=\text{sgn}(x)|x|^{1/3}$, the inverse is given by
\begin{eqnarray*}
\ell^{-1}(x,y)=\Big{(}x+(xy^2)^{1/3},y+(x^2y)^{1/3}\Big{)}.
\end{eqnarray*}
Then, if $\phi(x,y)=(\frac{x}{x+1},\frac{y}{y+1})$, $\ell^{-1}\circ\phi\circ\ell$ is given by the example (\ref{pvz7}). Of course, nor the
solution (\ref{pvz7}), neither $\phi(x,y)$ is continuous on the whole sphere $\widehat{\mathbb{R}^{2}}$.\\

 We directly verify that if $\ell$ is a non-degenerate linear transformation of $\mathbb{R}^{k}$, then
\begin{eqnarray*}
\ell^{-1}\circ\phi_{\m{a},Q}\circ\ell=\phi_{\ell^{-1}\m{a},Q\circ\ell}.
\end{eqnarray*}
The last identity shows that the Theorem can be given an alternative formulation, as presented above. In fact, let the solution of (\ref{funk}) be given by (\ref{sol}). Express $\ell=\ell_{1}\circ\ell_{2}$, where $\ell_{1}$ is a non-degenerate linear transformation. For a suitable $\ell_{1}$, $Q\circ\ell_{1}$ is a diagonal quadratic form given by the identity matrix. Further, the set of matrices which stabilize the diagonal quadratic form
$\sum_{i=1}^{k}x^{2}_{i}$ is the orthogonal group $O_{k}(\mathbb{R})$. With its help we can achieve that $\ell_{1}^{-1}\m{a}$ is a positive scalar multiple of any given non-zero vector. In particular, the multiple of the vector $\m{e}=(1,1,\ldots,1)$. Eventually, we write $\ell_{2}=\ell_{3}\circ\ell_{4}$, where $\ell_{3}(\m{x})=c\m{x}$ for a certain constant $c\in\mathbb{R}$. Thus, this shows that the Theorem can be given an alternative formulation.\\
\indent Similarly, consider the case $\m{a}=\infty$. Then, if $\ell$ is a non-degenerate linear transformation,
\begin{eqnarray*}
\ell^{-1}\circ\phi_{\m{c},L}\circ\ell=\phi_{\ell^{-1}\m{c},L\circ\ell}.
\end{eqnarray*}
Thus, if $L$ is a non-zero linear form, we can achieve that $L\circ\ell$ is a
 linear form $\sum_{i=1}^{k}x_{i}$. The set of matrices which stabilizes this form is a group $S_{k}(\mathbb{R})$ of stochastic matrices (that is, whose columns sum up to $1$). Thus, as before, using the scalar matrices as well, we can achieve that $\ell_{1}^{-1}\m{c}$ is any non-zero vector from the hyper-plane $\sum_{i=1}^{k}x_{i}=0$.\\

We now proceed with the verification that $\phi_{\m{a},Q}(\m{x})$ does indeed satisfy the functional equation (\ref{fuo}). This can be done directly; we will argue in a different way, which reveals the structure of this function far better. First, note that using form (\ref{alter}), we have $\phi_{\m{a},Q}(\m{x}z)=z\phi_{\m{a}z,Q}(\m{x})$.
This implies
\begin{eqnarray}
\phi^{z}_{\m{a},Q}(\m{x})=\phi_{\m{a}z,Q}(\m{x}).
\label{prop3}
\end{eqnarray}
\begin{prop} We have:
\begin{eqnarray}
\phi_{\m{a},Q}\circ\phi_{\m{b},Q}(\m{x})=\phi_{\m{a}+\m{b},Q}(\m{x}).
\label{prop2}
\end{eqnarray}
\end{prop}
\begin{proof}
Let $\m{x}\neq\m{0}$, $\m{b}Q(\m{x})+\m{x}\neq\m{0}$, and $Q(\m{x})\cdot Q(\m{b})+B(\m{x},\m{b})+1=\mathcal{T}$. Then the direct check shows that
\begin{eqnarray}
Q(\m{b}Q(\m{x})+\m{x})=\mathcal{T}\cdot Q(\m{x})\Rightarrow\mathcal{T}>0.\label{tarp}
\end{eqnarray}
Thus, let us define
\begin{eqnarray*}
\m{y}:=\phi_{\m{b},Q}(\m{x})=\Big{(}\m{b}Q(\m{x})+\m{x}\Big{)}\cdot\mathcal{T}^{-1}.
\end{eqnarray*}
After a straight substitution, and after multiplication of numerator and denominator by $\mathcal{T}$, we get
\begin{eqnarray*}
\phi_{\m{a},Q}(\m{y})&&=\Big{[}\m{b}Q(\m{x})+\m{x}+\m{a}Q
\Big{(}\m{b}Q(\m{x})+\m{x}\Big{)}\mathcal{T}^{-1}\Big{]}\\
&&\times\Big{[}Q(\m{b}Q(\m{x})+\m{x})\cdot Q(\m{a})\mathcal{T}^{-1}+B(\m{b}Q(\m{x})+\m{x},\m{a})
+\mathcal{T}\Big{]}^{-1}\\
&&\mathop{=}^{(\ref{tarp})}\Big{[}\m{b}Q(\m{x})+\m{x}+\m{a}Q(\m{x})\Big{]}\\
&&\times\Big{[}Q(\m{x})\cdot Q(\m{a})+B(\m{b}Q(\m{x})+\m{x},\m{a})
+\mathcal{T}\Big{]}^{-1}\\
&&=\Big{[}(\m{a}+\m{b})Q(\m{x})+\m{x}\Big{]}\\
&&\times\Big{[}Q(\m{x})\cdot Q(\m{a}+\m{b})+B(\m{x},\m{a}+\m{b})
+1\Big{]}^{-1}=\phi_{\m{a}+\m{b},Q}(\m{x}).
\end{eqnarray*}
Now, let $\m{x}\neq 0$, $\m{b}Q(\m{x})+\m{x}=\m{0}$. Without loss of generality, let $Q(\m{b})=1$. This implies $\m{x}=-\m{b}$, and
$\phi_{\m{b},Q}(-\m{b})=\infty$. Since $\phi_{\m{a},Q}(\infty)=\m{a}/Q(\m{a})$, we need to show that
\begin{eqnarray*}
\phi_{\m{a}+\m{b},Q}(-\m{b})=\frac{\m{a}}{Q(\m{a})},\quad\text{ if }Q(\m{b})=1.
\end{eqnarray*}
This is straightforward using (\ref{alter}).\end{proof}
Thus, we are ready to check that $\phi_{\m{a},Q}(\m{x})$ satisfies the functional equation (\ref{fuo}). Indeed,
\begin{eqnarray*}
\phi^{z_{1}}_{\m{a},Q}\circ\phi^{z_{2}}_{\m{a},Q}(\m{x})\mathop{=}^{(\ref{prop3})}\phi_{\m{a}z_{1},Q}\circ\phi_{\m{a}z_{2},Q}(\m{x})\mathop{=}^{(\ref{prop2})}
\phi_{\m{a}(z_{1}+z_{2}),Q}(\m{x})\mathop{=}^{(\ref{prop3})}\phi^{z_{1}+z_{2}}_{\m{a},Q}(\m{x}).
\end{eqnarray*}
Finally, we are left to verify the following
\begin{prop}
$\phi_{\m{a},Q}(\m{x})$ is a continuous function on $\widehat{\mathbb{R}^{ k}}$.
\end{prop}
\begin{proof}Assume that $Q(\m{a})=1$. The continuity is immediate in case $\m{x}=\infty$. Also, this is obvious if $\m{x}\neq-\m{a}$. Further, as noted above,
\begin{eqnarray*}
Q(\phi_{\m{a},Q}(\m{x}))=Q(\m{x})\cdot Q(\m{x}+\m{a})^{-1}.
\end{eqnarray*}
Thus, if $\m{x}\rightarrow-\m{a}$, the numerator of the above tends to $Q(-\m{a})=1$. The denominator tends to $0$. Thus,
\begin{eqnarray*}
Q(\phi_{\m{a},Q}(\m{x}))\mathop{\longrightarrow}_{\m{x}\rightarrow-\m{a}}\infty\Rightarrow \phi_{\m{a},Q}(\m{x})\mathop{\longrightarrow}_{\m{x}\rightarrow-\m{a}}\infty,
\end{eqnarray*}
and we are done.\end{proof}
\indent Let $\phi$ by any non-zero solution of (\ref{funk}), $\phi(\infty)=\m{a}$. In the functional equation (\ref{funk}) let us take the limit $\m{x}\rightarrow\infty$. Due to continuity, we obtain
\begin{eqnarray}
(1-z)\m{a}=\phi\Big{(}\frac{1-z}{z}\cdot\m{a}\Big{)}\Rightarrow\phi(\m{a}z)=\frac{z}{z+1}\cdot\m{a}\text{ for }z\in\widehat{\mathbb{R}}.
\label{aaa}
\end{eqnarray}
Let $\mathscr{A}={\rm Im}(\phi)$. From the very (\ref{funk}) we see that if $\m{x}\in\mathscr{A}$, $z\m{x}\in\mathscr{A}$.
Thus, if $\m{x}\notin\mathscr{A}$, $z\m{x}\notin\mathscr{A}$ for $z\neq 0$. Due to this single fact, and the facts (\ref{nulis}) and (\ref{aaa}) now it is easy to deduce that $\phi$ is surjective. Indeed, suppose the opposite. Let
$\mathscr{O}$ be the open ball whose intersection with $\mathscr{A}$ is empty. Then $\mathscr{A}$ does not intersect with $\mathscr{O}\times(0,\infty)\cup\mathscr{O}\times(-\infty,0)$. Now take the closed ball $\mathscr{B}$ around the origin of diameter $\|\m{a}\|/2$ (here we use the standard Euclidean norm). The segment $[-\m{a}/2,\m{a}/2]$ is mapped by $\phi$ bijectively onto $[-\m{a},\m{a}/3]$. The image of $\mathscr{B}$ under $\phi$ thus contains this segment and does not intersect two open cones $\mathscr{O}\times(0,\infty)$ and $\mathscr{O}\times(-\infty,0)$. But this is impossible due to the fact (\ref{nulis}) simply by a topological reason.\\

The functional equation (\ref{funk}) can be rewritten in the form
\begin{eqnarray}
\phi\Big{(}\phi(\m{x})z\Big{)}\frac{1}{z}=\phi\Big{(}\m{x}(z+1)\Big{)}\frac{1}{z+1}.
\label{it}
\end{eqnarray}
We see that $\lim_{z\rightarrow 0}\phi(\phi(\m{x})z)\frac{1}{z}$ exists and is equal to $\phi(\m{x})$. Since
${\rm Im}(\phi)=\widehat{\mathbb{R}^{ k}}$, this gives
\begin{eqnarray}
\lim\limits_{z\rightarrow 0}\frac{\phi(\m{x}z)}{z}=\m{x}\quad\text{ for }\m{x}\in\mathbb{R}^{k}.
\label{riba}
\end{eqnarray}
Suppose $\phi(\m{x})=\phi(\m{y})$. Then the functional equation (\ref{it}) shows that $\phi(\m{x}z)/z=\phi(\m{y}z)/z$. Taking the limit $z\rightarrow 0$ and using (\ref{riba}), we obtain $\m{x}=\m{y}$. Thus, we have proved the following
\begin{prop}
If a non-zero function $\phi$ satisfies the hypotheses of the Theorem, then $\phi$ is a bijection.
\end{prop}
Also, (\ref{it}) by induction gives
\begin{eqnarray}
\frac{1}{n}\cdot\phi(n\m{x})=\underbrace{\phi\circ\cdots\circ\phi}\limits_{n}(\m{x})\text{ for }n\in\mathbb{N};\quad
-\phi(\m{x})=\phi^{-1}(-\m{x}).
\label{it2}
\end{eqnarray}
Here the sign $``-1"$ refers to the inverse function. Curiously, as we see, this exactly coincides with the notation $\phi^{z}(\m{x})$ for $z=-1$.\\
\indent The following proposition is not needed for the final proof of the Theorem; nevertheless, we include it for the better understanding of the functional equation (\ref{funk}).
\begin{prop}
Let $\gamma$ and $\chi$ be two continuous solutions to (\ref{funk}). If $\gamma$ and $\chi$ commute, then $\gamma\circ\chi$
is a solution to $(\ref{funk})$ as well.
\end{prop}
\begin{proof} First, we will show that in this case $\gamma$ and $\chi^{z}$ commute too. Let $z=\frac{m}{n}$ be a positive rational number.
Then
\begin{eqnarray*}
\chi^{\frac{m}{n}}\circ\gamma(\m{x})&=&\frac{n}{m}\chi\Big{(}\frac{m}{n}\gamma(\m{x})\Big{)}\mathop{=}^{(\ref{it2})}
\frac{n}{m}\chi\Big{(}m\underbrace{\gamma\circ\cdots\circ\gamma}\limits_{n}(\frac{\m{x}}{n})\Big{)}\\
&=&n\underbrace{\chi\circ\cdots\circ\chi}\limits_{m}\circ\underbrace{\gamma\circ\cdots\circ\gamma}\limits_{n}(\frac{\m{x}}{n})=
n\underbrace{\gamma\circ\cdots\circ\gamma}\limits_{n}\circ\underbrace{\chi\circ\cdots\circ\chi}\limits_{m}(\frac{\m{x}}{n})\\
&=&\gamma\Big{(}\frac{n}{m}\chi(\frac{m}{n}\m{x})\Big{)}=\gamma\circ\chi^{\frac{m}{n}}(\m{x}).
\end{eqnarray*}
Analogously, the second identity of (\ref{it2}) shows that $\gamma$ and $\chi^{\frac{m}{n}}$ commute for negative rational $\frac{m}{n}$. Since, according to Proposition \ref{prop1},
$\chi^{\frac{m}{n}}$ is also a solution to (\ref{funk}), we see that
$\gamma^{z_{1}}$ and $\chi^{z_{2}}$ commute for all pairs of rational numbers $z_{1}$, $z_{2}$. Due to continuity, they commute for all $z_{1},z_{2}\in\mathbb{R}$. Finally, this gives
\begin{eqnarray*}
(\gamma\circ\chi)^{z_{1}}\circ(\gamma\circ\chi)^{z_{2}}\mathop{=}^{(\ref{trans})}
\gamma^{z_{1}}\circ\chi^{z_{1}}\circ\gamma^{z_{2}}\circ\chi^{z_{2}}
=\gamma^{z_{1}}\circ\gamma^{z_{2}}\circ\chi^{z_{1}}\circ\chi^{z_{2}}=\gamma^{z_{1}+z_{2}}\circ\chi^{z_{1}+z_{2}}
\mathop{=}^{(\ref{trans})}(\gamma\circ\chi)^{z_{1}+z_{2}}.
\end{eqnarray*}
This proves the statement.\end{proof}

Suppose, as before, that $\phi$ satisfies the hypotheses of the Theorem. Let us define
\begin{eqnarray*}
 \mathscr{V}(\m{x})=\Big{\{}\frac{1}{z}\cdot\phi(\m{x}z):z\in\widehat{\mathbb{R}}\Big{\}}.
\end{eqnarray*}
Then from the functional equation and other properties we inherit that these sets for different $\m{x}$ either coincide, or have only a single common point $\m{x}=\m{0}$. Indeed, if
\begin{eqnarray*}
\frac{1}{z_{1}}\cdot\phi(\m{x}_{1}z_{1})=\frac{1}{z_{2}}\cdot\phi(\m{x}_{2}z_{2})\Rightarrow
\m{x}_{1}=\frac{1}{z_{2}-z_{1}}\cdot\phi(\m{x}_{2}(z_{2}-z_{1})).
\end{eqnarray*}
Thus,
\begin{eqnarray*}
\mathscr{V}(\m{x})=\mathscr{V}(\phi(\m{x}));\quad \phi(\m{x})\in\mathscr{V}(\m{x});\quad\m{x}\in\mathscr{V}(\m{x});\quad \m{0}\in\mathscr{V}(\m{x}).
\end{eqnarray*}
There is a continuous bijection
\begin{eqnarray*}
T:\mathscr{V}(\m{x})\rightarrow\widehat{\mathbb{R}},\quad \frac{1}{z}\cdot\phi(\m{x}z)\mapsto z.
\end{eqnarray*}
In particular,
\begin{eqnarray*}
T(\m{x})=0;\quad T(\phi(\m{x}))=1;\quad T(\m{0})=\infty;\quad T\circ\phi\circ T^{-1}(z)=z+1.
\end{eqnarray*}
Let us call $\mathscr{V}(\m{x})$ {\it the orbit of} $\m{x}$. Let us call a set $\mathscr{C}$ {\it a representation set for} $\phi$, if it is homeomorphic to a sphere $\mathbb{S}^{k-1}$, it contains the origin, and a single non-zero representative from each other orbit. We need the following
\begin{prop}
Let $\phi$ be a solution to (\ref{funk}), $\m{a}$ is finite. Then $\m{x}$ and $\phi(\m{x})$ are not collinear, provided that $\m{x}$ is not a multiple of $\m{a}$.
\label{prop9}
\end{prop}
\begin{proof}As is clear from (\ref{funk}), the equation
\begin{eqnarray*}
\m{x}=\phi(\m{x}z) \frac{1-z}{z}
\end{eqnarray*}
can have a solution only if $\phi(\m{x})=\infty$; that is, $\m{x}=-\m{a}$. Replacing $\m{x}$ with $\m{x}/z$, we see that $\m{x}=\phi(\m{x})(1-z)$ has a solution only iff $\m{x}$ is collinear with $\m{a}$.\end{proof}

Thus, for a given $\m{x}$ not parallel with $\m{a}$, any line $z\m{x}$, $z\in\mathbb{R}$, cannot contain two different non-zero points belonging to the same orbit. Indeed, assume the opposite. Then these points are $\m{x}_{0}$ and $\frac{1}{z}\phi(\m{x}_{0}z)$. By the assumption, they are collinear; that is, $y\m{x}_{0}=\frac{1}{z}\phi(\m{x}_{0}z)$. But this gives $y(\m{x}_{0}z)=\phi(\m{x}_{0}z)$, which contradicts Proposition \ref{prop9}.\\

The final ingredient into the proof of the Theorem is a resolution of the following
\begin{quest}
Let a continuous function $\phi$ satisfies (\ref{funk}), $\m{a}$ is finite. We ask whether there exists a representation set for $\phi$, call it $\mathscr{C}$, which is invariant under homothety: $\mathscr{C}=z\mathscr{C}$, $z\in\mathbb{R}\setminus\{0\}$.
\label{que}
\end{quest}
\indent{\it Remark 4. }The case $k=1$ is trivial (all vectors are multiples of $\m{a}$). Further, if $k=2$, then, according to the above observation, every line through the origin, which is not identical to $\mathbb{R}\m{a}$, intersects each orbit in not more than a single point. Since $\phi^{z}(\m{x})=(\m{a}+o(1))/z$ as $z\rightarrow\pm\infty$, each such line is a representation set. Here the origin represents itself and the point at infinity represents the orbit of $\m{a}$. Unfortunately, our efforts to answer affirmatively and rigorously this question for $k\geq 3$ have failed. Nevertheless, for a particular $\phi$ a representation set can be constructed explicitly. For example, we have
\begin{prop} A compactification of the hyperplane
\begin{eqnarray*}
\mathscr{C}=\{\m{x}=(x_{1},x_{2},\ldots,x_{k}):\sum\limits_{i=1}^{k}x_{i}=0\}
\end{eqnarray*}
is a representation set for $\phi_{1}$ (as given by (\ref{basic})).
\end{prop}
\begin{proof}The point $\m{0}$ represents its own orbit. If $\m{y}$ is proportional to the vector $(1,1,\ldots,1)$, the representative of its orbit is the point at infinity. Suppose, $\m{y}=(y_{1},y_{2},\ldots,y_{k})\neq\m{0}$, and not all $y_{j}$ are equal.
And so, we need to show that the system of equations
\begin{eqnarray*}
\left\{\begin{array}{c@{\qquad}l}\displaystyle y_{j}=\frac{z\sum\limits_{i=1}^{k}x^{2}_{i}+k\cdot x_{j}}
{\sum\limits_{i=1}^{k}(zx_{i}+1)^2} & j=1,2,\ldots,k;
\\ \sum\limits_{i=1}^{k}x_{i}=0; \end{array}\right.
\end{eqnarray*}
has a unique solution $\{z,x_{1},x_{2},\ldots,x_{k}\}$. Indeed, let $\sum_{i=1}^{k}x_{i}^{2}=X$, $\sum_{j=1}^{k}y_{j}=Y$. If $Y=0$, the solution is
given by $z=0$, $y_{j}=x_{j}$. Suppose $Y\neq0$. Then $z\neq0$. Thus,
\begin{eqnarray}
z^{2}y_{j}X=zX+kx_{j}-ky_{j},\quad j=1,2,\ldots,k.
\label{ygrek}
\end{eqnarray}
Summing this over $j=1,2,\ldots,k$, we obtain
\begin{eqnarray*}
z^{2}XY=kzX-kY\Rightarrow X=\frac{kY}{kz-z^{2}Y}.
\end{eqnarray*}
Therefore, (\ref{ygrek}) can be rewritten as
\begin{eqnarray*}
x_{j}=y_{j}-Y\cdot\frac{1-zy_{j}}{k-zY}.
\end{eqnarray*}
The only indeterminate left is $z$, and it is found from the equation $\sum_{j=1}^{k}x^{2}_{j}=X$. In other words,
\begin{eqnarray*}
\sum\limits_{j=1}^{k}\Big{(}y_{j}-Y\cdot\frac{1-zy_{j}}{k-zY}\Big{)}^2=\frac{kY}{kz-z^{2}Y}\Rightarrow
\sum\limits_{j=1}^{k}(ky_{j}-Y)^{2}=kY\cdot\frac{k-zY}{z}.
\end{eqnarray*}
Obviously, this equation has a unique solution in $z$ if $y_{j}\neq\frac{Y}{k}$ for at least one $j$, which by our assumption does hold (of course, there exists a unique solution in case all $y_{j}$ are equal, but then $z=\frac{k}{Y}$, and our reasoning is invalid). \end{proof}

\begin{proof}{\it (The Theorem).} Suppose $k=2$, $\phi$ is a non-zero continuous solution to (\ref{funk}) with $\m{a}$ finite. Fix any known non-zero continuous solution of (\ref{funk}) whose exceptional vector is also finite, call this solution $\chi$ (let us take, for example, $\chi=\phi_{1}$ as is given by (\ref{basic})). Choose two sets $\mathscr{C}_{\phi}$ and $\mathscr{C}_{\chi}$, which are invariant under homothety and which are representation sets for $\phi$ and $\chi$ respectively. According to remark after the Question \ref{que}, this can be done. Since they are both homeomorphic to a sphere $\mathbb{S}^{k-1}$, let
$\Delta:\mathscr{C}_{\chi}\mapsto\mathscr{C}_{\phi}$ be a homothetic homeomorphism. This, of course, implies $\Delta(\m{0})=\m{0}$.
Let us define the following function:
\begin{eqnarray*}
\ell\Big{(}\frac{1}{z}\cdot\chi(\m{x}z)\Big{)}=\frac{1}{z}\cdot\phi(\Delta(\m{x})z),\quad \m{x}\in\mathscr{C}_{\chi},\quad z\in\widehat{\mathbb{R}}.
\end{eqnarray*}
(As an aside, taking the limit the limit $z\rightarrow 0$ gives $\ell(\m{x})=\Delta(\m{x})$ for $\m{x}\in\mathscr{C}_{\chi}$).
By the very construction, the function $\ell$ is well-defined on the whole $\widehat{\mathbb{R}^{k}}$. Second, this is obviously a continuous function. By the very definition of $\mathscr{C}_{\chi}$ and $\mathscr{C}_{\phi}$, it is a bijection. Moreover, it is a homothetic function.  Indeed, for $y\in\mathbb{R}$, $z\in\mathbb{R}$, $\m{x}\in\mathscr{C}_{\chi}$, we have:
\begin{eqnarray*}
\ell\Big{(}\frac{y}{z}\cdot\chi(\m{x}z)\Big{)}=\ell\Big{(}\frac{y}{z}\cdot\chi\Big{(}\m{x}y\cdot\frac{z}{y}\Big{)}\Big{)}=
\frac{y}{z}\phi\Big{(}\Delta(\m{x}y)\cdot\frac{z}{y}\Big{)}=\frac{y}{z}\phi(\Delta(\m{x})z)=
y\ell\Big{(}\frac{1}{z}\cdot\chi(\m{x}z)\Big{)}.
\end{eqnarray*}
(We have used a fact that for $\m{x}\in\mathscr{C}_{\chi}$, $\m{x}y\in\mathscr{C}_{\chi}$). Finally, let $\m{y}=\frac{1}{z}\cdot\chi(\m{x}z)$. Then
\begin{eqnarray*}
\ell^{-1}\circ\phi\circ\ell(\m{y})&=&\ell^{-1}\circ\phi\Big{[}\frac{1}{z}\cdot\phi(\Delta(\m{x})z)\Big{]}
\mathop{=}^{(\ref{fuo})}\ell^{-1}\Big{[}\frac{1}{z+1}\cdot\phi\Big{(}\Delta(\m{x})(z+1)\Big{)}\Big{]}\\
&=&\frac{1}{z+1}\cdot\chi\Big{(}\m{x}(z+1)\Big{)}=\chi(\m{y}).
\end{eqnarray*}
This shows that any solution of (\ref{funk}) can be obtained by a conjugation of another solution with a homothetic function, and this finishes the first half of the Theorem. As already mentioned, the presented proof is valid for finite $\m{a}$. The proof, which covers the fourth case of the Theorem (in the alternative formulation), is completely analogous. \end{proof}

\section{$1$-dimensional case}
We finish with a demonstration that, provided we limit ourselves to a positive half-line, we can give the full solution in $1$-dimensional case.
\begin{prop}
Let $\mathbb{R}_{+}=(0,\infty)$. All functions $f:\mathbb{R}_{+}\mapsto\mathbb{R}_{+}$, which
satisfy the functional equation (\ref{funkk}), are given by
$f(x)=\frac{x}{Cx+1}$, where $C\geq 0$ is a fixed constant.
\label{prop10}
\end{prop}
\begin{proof} If $z=\frac{u}{x}$, $u<x$,
the functional equation can be rewritten as
\begin{eqnarray}
\Big{(}1-\frac{u}{x}\Big{)}f(x)=f\Big{(}f(u)\frac{x-u}{u}\Big{)}.
\label{second}
\end{eqnarray}
Suppose for some $u$ we have $f(u)>u$. Therefore, $f(u)=Lu$ for
some real $L>1$. Then the equation in $x$, $f(u)\frac{x-u}{u}=x$,
that is, $L(x-u)=x$, has a positive real solution
$x_{0}=\frac{Lu}{L-1}>u$, and then
$(1-\frac{u}{x_{0}})f(x_{0})=f(x_{0})$; thus, since $f(x_{0})>0$,
we obtain a contradiction. Therefore, $f(x)\leq x$ for all
positive $x$. Then the initial equation gives
\begin{eqnarray}
(1-z)f(x)=f\Big{(}f(xz)\frac{1-z}{z}\Big{)}\leq
f(xz)\frac{1-z}{z};\label{nel}
\end{eqnarray}
therefore, $\frac{f(x)}{x}\leq\frac{f(xz)}{xz}$ for $0<z<1$, and
this implies that $\frac{f(x)}{x}\leq 1$ is monotonically decreasing
(in a not strict sense). Suppose, for some $u_{0}$ we have an
equality $f(u_{0})=u_{0}$. Then $f(x)=x$ for all $x\leq u_{0}$.
Let in the initial equation $z\in[\frac{1}{3},\frac{2}{3}]$, and
$x\leq \frac{3}{2}u_{0}$. Then $xz\leq u_{0}$, and hence
$f(xz)=xz$. Further, $x(1-z)\leq u_{0}$ as well, and in the same
fashion $f(x(1-z))=x(1-z)$. Therefore, $(1-z)f(x)=(1-z)x$, for all
$x$, $x\leq \frac{3}{2}u_{0}$. By induction, $f(x)=x$ for all
positive $x$, $x\leq (\frac{3}{2})^{N}u_{0}$, $N\in\mathbb{N}$,
and consequently, $f(x)=x$ for all real $x>0$. This function
satisfies the functional equation. Suppose, $f(x)<x$ for all real
$x>0$. Returning to the initial equation, we see that in
($\ref{nel}$) the last inequality is strict, and in this case
$\frac{f(x)}{x}$ is strictly decreasing. Now fix two arbitrary positive $u$
and $w$. In the second form of the equation, that is (\ref{second}), let us find $x$, for
which $f(u)\frac{x-u}{u}=w$. Thus, $x=\frac{uw}{f(u)}+u$, and
since $\frac{f(x)}{x}(x-u)=f(w)$, we obtain
\begin{eqnarray*}
\frac{f(x)}{x}=\frac{f(u)}{u}\cdot\frac{f(w)}{w}.
\end{eqnarray*}
From the symmetry, if $y=\frac{uw}{f(w)}+w$, we get
$\frac{f(y)}{y}=\frac{f(u)}{u}\cdot\frac{f(w)}{w}$, and this gives
$\frac{f(x)}{x}=\frac{f(y)}{y}$. Since the last function is
strictly decreasing, this can happen only iff $x=y$.
Therefore,
\begin{eqnarray*}
\frac{uw}{f(u)}+u=\frac{uw}{f(w)}+w, \quad u,w>0.
\end{eqnarray*}
This implies that $\frac{1}{f(u)}-\frac{1}{u}$ is independent of
$u$, and thus is equal to $C$, $C>0$ is a positive constant.
Ultimately, $f(x)=\frac{x}{Cx+1}$, $C>0$. The function $f(x)=x$ is
a special case of the last with $C=0$. The check shows that these
functions in fact satisfy the functional
equation. \end{proof}
If we know {\it a priori} that the solution of this functional equation is continuous, and, moreover,
$f(x)\rightarrow A>0$, as $x\rightarrow\infty$, we may consider the functional equation as $x\rightarrow\infty$. This gives
\begin{eqnarray*}
(1-z)A=f\Big{(}A(1-z)/z\Big{)},
\end{eqnarray*}
and the answer follows. Surprisingly, it appears that to solve this functional equation we do not need any
assumption on $f$. The presented solution only assumes that $f$ is a set-theoretic function $f:\mathbb{R}_{+}\mapsto\mathbb{R}_{+}$.\\

\par\bigskip

\noindent

\noindent {\sc Giedrius Alkauskas}, Institute of Mathematics, Department of Integrative Biology,
Universit\"{a}t f\"{u}r Bodenkultur Wien, Gregor Mendel-Stra{\ss}e 33, A-1180 Wien, Austria, \&\\
Vilnius University, Department of Mathematics and Informatics, Naugarduko 24, LT-03225 Vilnius, Lithuania.\\
{\tt giedrius.alkauskas@gmail.com}\\

\smallskip
\end{document}